\documentclass[12pt,leqno]{article}
\usepackage{amssymb, amscd, amsmath, amsthm}

\allowdisplaybreaks

\def\ve{\varepsilon}
\def\beq{\begin{equation}}
\def\eeq{\end{equation}}

\newtheorem{theorem}{Theorem}
\def\cite#1{{\rm [#1]}}

\newtheorem{lemma}{Lemma}
\newtheorem{corollary}{Corollary}
\theoremstyle{definition}

\numberwithin{theorem}{section}
\numberwithin{corollary}{section}
\numberwithin{lemma}{section}
\numberwithin{equation}{section}

\begin{document}

\title{Oscillation of partial sums of the M\"obius function and zeros of Riemann's zeta function}

\author{by\\
J\'anos Pintz\thanks{Supported by the Hungarian  National Research, Development and
Innovation Office (NKFIH) grants No.\ KKP133819, K147153 and EXCELLENCE 151341.}}

\date{}


\maketitle

\begin{abstract}
The oscillation of $M(x)$, the partial sum of the M\"obius function has been
in the focus of researchers in the theory of primes since the famous
conjecture of Mertens in 1905 (formulated in a weaker form by Stieltjes
in 1885 in a letter to Hermite).
The average order of the modulus of $M(x)$ in an interval of type $[0,Y]$ is clearly in connection with the distribution of primes.
The author proved at the beginning of 1980's that this average tends unconditionally to infinity with $Y$.
The present work shows that (similarly to the average order of the error term of the
Prime Number Theorem) this average agrees with great accuracy for large
values of $Y$ with the largest error term of the Riemann--von Mangoldt
prime number formula for the value $Y$.
\end{abstract}



\section{Introduction}
\label{sec:1}
Ingham \cite{Ing1932, Theorem 22} showed nearly hundred years ago that a rather general zero-free region of Riemann's zeta function implies an upper bound for the error term of the prime number theorem.
Let us consider the error term in the form
\beq
\label{eq:1.1}
\Delta(x) = \psi(x) - x, \ \ \ \psi(x) = \sum_{n \leq x} \Lambda(n), \ \ \ \Lambda(n) = \begin{cases}
\log p &\text{ if } n = p^m,\\
0 &\text{ otherwise}\end{cases}
\eeq
where $p$ denotes always primes, $m, n$ positive integers
and non-trivial zeros of $\zeta(s)$, Riemann's zeta function, will be written in the form $\varrho = \beta + i\gamma$.
Then a shortened form of the famous Riemann--von Mangoldt formula shows the connection between
$\Delta(x)$ and the distribution of zeros:
\beq
\label{eq:1.2}
\Delta(x) = -\sum_{|\gamma| \leq T} \frac{x^{\varrho}}{\varrho} + O\left(\frac{x}{T} \log^2 xT\right) \ \ \text{ for } \ T \leq x.
\eeq

A classical consequence of \eqref{eq:1.2} and some other well-known properties of $\zeta(s)$ is with the notation
\beq
\label{eq:1.3}
\vartheta = \sup_{\varrho~:~\zeta(\rho)=0} \beta,
\eeq
the relation
\beq
\label{eq:1.4}
\Delta(x) = O(x^\vartheta \log^2 x)
\eeq
(proved by von Koch \cite{Koc1904} for $\vartheta = 1/2$, the case of the Riemann Hypothesis (RH)).
This can even be sharpened to
\beq
\label{eq:1.5}
\Delta(x) = O(x^\vartheta) \ \ \text{ if }\ \vartheta > 1/2,
\eeq
using the classical density theorem of Carlson \cite{Car1920}
\beq
\label{eq:1.6}
N(\sigma, T) := \sum_{\substack{\varrho; \beta \geq \sigma\\
|\gamma| \leq T}} 1 = O\bigl(T^{4\sigma(1 - \sigma)}\log^C T\bigr)
\eeq
(where $C$ and later $c^*$ etc. denote in the following a generic absolute constant, not necessarily the same at different occurrences).

Ingham's \cite{Ing1932} mentioned result was specially interesting for the case $\vartheta = 1$ and sounded as follows.

\medskip
\noindent
{\bf Theorem A.}
{\it Suppose
\beq
\label{eq:1.7}
\zeta(s) \neq  0 \ \ \text{ for } \ \ \sigma > 1 - \eta(t) \ \ \ (s = \sigma + it)
\eeq
where $\eta(t)$ is for $t \geq 1$ a real function with
\begin{align}
\label{eq:1.8}
\eta(t) &\in C^1 [1, \infty),\\
\label{eq:1.9}
\eta'(t) &\leq 0,\\
\label{eq:1.10}
\eta'(t) &\to 0 \ \ \text{ as } \ \ t \to \infty,\\
\label{eq:1.11}
\frac{1}{\eta(t)} &= 0 (\log t) \ \ \text{ as } \ \ t \to \infty.
\end{align}
Let $\varepsilon$ be a fixed number satisfying $0 < \varepsilon < 1$ and let
\beq
\label{eq:1.12}
\omega_\eta(x) := \inf_{t \geq 1} \bigl(\eta(t) \log x + \log t\bigr).
\eeq
Then
\beq
\label{eq:1.13}
\Delta(x) = \sum_{p^n \leq x} \log p - x \ll x e^{-(1/2)(1 - \varepsilon) \omega_\eta(x)}.
\eeq}

\medskip
In the important special case
\beq
\label{eq:1.14}
\zeta(s) \neq 0 \ \ \text{ for } \ \ \sigma > 1 - \frac{c_1}{\log^\alpha t}, \ \ \ t > t_0,
\eeq
Ingham's theorem yields with a $c_2$ depending on $c_1$ and $\alpha$
\beq
\label{eq:1.15}
\Delta(x) \ll x \exp \left( - c_2 \log^{1/(1 + \alpha)} x \right).
\eeq
It is an important problem whether or not the inverse implication \eqref{eq:1.15} $\Rightarrow$ \eqref{eq:1.14} is true.
This question was answered in the affirmative by Tur\'an \cite{Tur1950} in 1950.
In the proof a crucial role was played by the power-sum method of Tur\'an.
Tur\'an's result was extended by W. Sta\'s \cite{Sta1961} for a wider class of functions $\eta(t)$.
Their works implied, however, \eqref{eq:1.14} with a value $c_1' < c_1 / 80$.

The author \cite{Pin1980} succeeded in proving Ingham's theorem in the sharper form
\beq
\label{eq:1.16}
\Delta(x) \ll xe^{-(1 - \ve)\omega_\eta(x)}
\eeq
in case of an arbitrary continuous decreasing function $\eta(t)$.
It was further shown in \cite{Lit1914} that if $\eta(t)$ satisfies some conditions and if \eqref{eq:1.16} is true with $1 + \ve$ in place of $1 - \ve$ then for $t > t_0$ we have \eqref{eq:1.7}.

Since $|\Delta(x)|$ takes also small values, e.g.\ $\Delta(x_n) = 0$ for a suitable sequence $x_n \to \infty$, in order to study the oscillatory behaviour of the error term we consider the functions
\beq
\label{eq:1.17}
S(x) = \max_{u \leq x} |\Delta(u)|, \ \ \ D(x) = x^{-1} \int\limits_0^x |\Delta(u)|du.
\eeq
The aim of our work was to show the existence of a real function $\omega(x)$ -- depending in a simple way on the distribution of zeros of $\zeta(s)$ -- which determines the functions $S(x)$ and $D(x)$ with high accuracy.
An interesting consequence of our investigations was that the asymptotic behaviour of $S(x)$ and $D(x)$ turned out to be nearly the same.

The classical general $\Omega$-result before Tur\'an \cite{Tur1950} was only that of Phragm\'en \cite{Phr1891} which asserted for any zero $\varrho_0 = \beta_0 + i\gamma_0$
\beq
\label{eq:1.18}
\Delta(x) = \Omega\bigl(x^{\beta_0 - \varepsilon}\bigr) \ \ \text{ for any } \ \varepsilon > 0
\eeq
and was completely ineffective.
Supposing RH, Littlewood \cite{Lit1914} could prove
\beq
\label{eq:1.19}
\Delta(x) = \Omega\left(\sqrt{x} \log_3 x\right),
\eeq
where $\log_\nu x$ denotes the $\nu$ times iterated logarithmic function.
\eqref{eq:1.18} implies clearly
\beq
\label{eq:1.20}
\Delta(x) = \Omega(x^{\vartheta - \varepsilon}).
\eeq

This, together with \eqref{eq:1.4}, indicates the connection between the oscillation of $\Delta(x)$ and the
distribution of zeta-zeros, but is very weak if $\vartheta = 1$.

In our works \cite{Pin1980}, \cite{Pin2017} we succeeded to establish a sharp connection in the above direction, valid for all $\vartheta \geq 1/2$ (although the difficult part was the case $\vartheta = 1$).
In order to formulate the results of \cite{Pin2017} we introduce some notation, which -- or their analogues -- will be used in our present theorems as well.


Let us associate the zero-free domain of $\zeta(s)$ with the simple real function $(x > 1)$
\beq
\label{eq:2.1}
\omega(x) = \min_{\varrho; \gamma > 0} (\delta \log x + \log \gamma)
\eeq
where we denote the non-trivial zeta-zeros by
\beq
\label{eq:2.2}
\varrho = \beta + i \gamma = 1 - \delta + i \gamma.
\eeq
Our main result (with the notation \eqref{eq:1.17}) is the relation
\beq
\label{eq:1.23}
\log \frac{x}{S(x)} \sim \log \frac{x}{D(x)} \sim \omega(x) \ \ \text{ as }\ x \to \infty .
\eeq
The heuristic meaning of \eqref{eq:1.23} is that $S(x)$ and $D(x)$ have approximately the same order of magnitude as
\beq
\label{eq:1.24}
Z(x) = \max_{\varrho; \gamma > 0} \frac{x^\beta}{\gamma} = \frac{x}{e^{\omega(x)}}
\eeq
which is essentially (apart from the factor $|\varrho| / \gamma$) the modulus of the largest term in the following explicit formula of $\Delta(x)$:
\beq
\label{eq:2.5}
\Delta(x) = - \sum_{\substack{\varrho\\ |\gamma| \leq x}} \frac{x^\varrho}{\varrho} + O (\log^2 x),
\eeq
the Riemann-von Mangoldt formula \eqref{eq:1.2} with $T=x$.
In order to complete \eqref{eq:1.23} we defined further
\beq
\label{eq:1.26}
W(x) = \sum_{\substack{\varrho\\ |\gamma| \leq x}} \frac{x^\beta}{|\gamma|}, \ \ \ \omega_W(x) = \log \frac{x}{W(x)}
\eeq
and analogously let $\omega_D(x)$, $\omega_S(x)$ defined by
\beq
\label{eq:1.27}
\omega_D(x) := \log \frac{x}{D(x)},
\eeq
\beq
\label{eq:1.28}
\omega_S(x) := \log \frac{x}{S(x)}.
\eeq

We showed in \cite{Pin2017}

\medskip
\noindent
{\bf Theorem B.}
{\it With the above notation we have
\beq
\label{eq:1.28}
\omega_W(x) \sim \omega_S(x) \sim \omega_D(x) \sim \omega(x)\ \ \text{ as } \ x \to \infty.
\eeq}

Since $\omega(x)/\log x = \min\limits_{\varrho;\gamma > 0} (\delta + \log \gamma/\log x) \to 1 - \vartheta$ as $x \to \infty$ we obtain from Theorem B also $1 - \omega_A(x)/\log x \to \vartheta$ for $A(x) = W(x)$, $S(x)$ and $D(x)$.
Since $\vartheta > 0$ we can formulate this as

\medskip
\noindent
{\bf Theorem C.}
{\it With the above notation we have
\beq
\label{eq:1.29}
\log W(x) \sim \log S(x) \sim \log D(x) \sim \log Z(x) \sim  \vartheta \log x.
\eeq}

\medskip
Theorem B is a far-reaching generalization and sharpening of the results of Ingham \cite{Ing1932} about upper bounds for the error term and that of Tur\'an \cite{Tur1950} about oscillation of the error term as detailed in Theorems 6, 7 and Corollary 2 of our work \cite{Pin2017}.

The meaning of Theorems B and C is that the maximal and average order of the error term coincides with high accuracy with the maximal term of the explicit formula and with the sum of all terms in the explicit formula.
In order to formulate the results exactly we do not need the use of zero-free regions (and their properties), it is sufficient to use the simple functions $W(x)$ or $Z(x)$ (see \eqref{eq:1.24} and \eqref{eq:1.26}) in place of the Riemann--Von Mangoldt formula \eqref{eq:1.2}.

\section{Partial sums of the M\"obius function}
\label{sec:2}

Our present goal is to prove analogous sharp results for the partial sums
\beq
\label{2.1}
M(x) = \sum_{n \leq x} \mu(n)
\eeq
in place of the error term $\Delta(x)$ of the prime number theorem.
However, the case of $M(x)$ is much more difficult, since, instead of $\zeta'(s)/\zeta(s)$ we have to work with $1/\zeta(s)$, due to
\beq
\label{eq:2.2}
M(x) = \frac1{2\pi i} \int\limits_{(2)} \frac{x^s}{s\zeta(s)} ds.
\eeq
Thus, even supposing that all zeros are simple we have instead of \eqref{eq:1.2} the much more uncomfortable singularities
\beq
\label{eq:2.3}
\frac{x^\varrho}{\varrho \zeta'(\varrho)}
\eeq
even if we forget about the convergence of the series with terms of the form \eqref{eq:2.3}.

The extra difficulties compared with the oscillation of the error term could be seen from the fact that the best upper bound under the RH is
\beq
\label{eq:2.4}
M(x) = \sqrt{x} \exp \bigl((\log x)^{1/2} \log_2^{14} x \bigr)
\eeq
due to Soundararajan \cite{Sou2009}.
This improved the earlier exponent $39/61$ of $\log x$ proved by H. Maier and H. L. Montgomery \cite{MM2009}.
The classical result of Littlewood \cite{Lit1912} was just $M(x) = O\bigl(x^{1/2 + \varepsilon}\bigr)$ on RH.
Concerning the oscillation of $M(x)$, Mertens conjectured
\beq
\label{eq:2.5}
|M(x)| \leq \sqrt{x}
\eeq
which was disproved only in 1985 by Odlyzko and te Riele \cite{OR1985} using among others large scale computations concerning zeta-zeros.

We even do not know today whether $|M(x)| \leq 2\sqrt{x}$ for all $x$ and the weaker analogue of Littlewood's result \eqref{eq:1.18},
\beq
\label{eq:2.6}
\limsup_{x \to \infty} \frac{M(x)}{\sqrt{x}} = \infty,
\eeq
seems to be hopeless. (Ingham \cite{Ing1942} proved it supposing RH and linear independence of the positive imaginary parts of zeta-zeros over the rationals.)
We remark that if the RH is false, i.e., $\vartheta > 1/2$, then
\beq
\label{eq:2.7}
M(x) = \Omega\bigl(x^{\vartheta - \varepsilon}\bigr)
\eeq
follows from the mentioned theorem of Phragm\'en \cite{Phr1891}.

We mention further that the first result showing
\beq
\label{eq:2.8}
D_M(Y) := \frac{1}{Y} \int\limits_1^Y |M(x)|dx \longrightarrow \infty \ \ \text{ as }\ Y \to \infty
\eeq
was proved by the author \cite{Pin1982/83} with the estimate that for any $\varrho_0 > 0$
\beq
\label{eq:2.9}
D_M(Y) \geq \frac{Y^{\beta_0}}{6|\gamma_0|^3} \ \ \text{ for }\ Y > Y_0 = e^{|\gamma_0| + 4}.
\eeq

After this estimate it is relatively easy to show the analogue of the weaker Theorem C (see \eqref{eq:1.29}).
However, using \eqref{eq:2.9} we cannot prove the complete analogue of Theorem 3 (see \eqref{eq:1.28}) in the critical case $\vartheta = 1$.

In order to formulate our present results we repeat the definitions \eqref{eq:1.26}, \eqref{eq:1.24} and \eqref{eq:2.8} of $W(x), \omega_W(x), Z(x), \omega_Z(x)$ and $D_M(x)$, and define also $\omega_{D_M}(x)$ as follows.
\beq
\label{eq:2.10}
W(x) := \sum_{\varrho; |\gamma| \leq x} \frac{x^\beta}{\gamma}, \ \ \ \omega_W(x) := \log \frac{x}{W(x)},
\eeq
\beq
\label{eq:2.11}
Z(x) := \max_{\varrho; \gamma > 0} \frac{x^\beta}{\gamma}, \ \ \ \omega(x) := \log \frac{x}{Z(x)}.
\eeq
\beq
\label{eq:2.12}
D_M(x) := \frac{1}{x} \int\limits_0^x |M(u)|du, \ \ \ \omega_{D_M}(x) = \log \frac{x}{D_M(x)}.
\eeq

To complete the picture we will further investigate the maximum order of $M(x)$, that is,
\beq
\label{eq:2.13}
S_{M, \delta}(x) = \max_{x^{1 - \delta} \leq u \leq x} |M(u)|,\ \ \omega_{S_{M,\delta}}(x) = \log \frac{x}{S_{M, \delta}(x)}
\eeq
with an arbitrary fixed small $\delta > 0$,
and show (analogously to the case of the error term $\Delta(x)$ of the PNT, see Section \ref{sec:1})
that it coincides with high accuracy with the average order $D_M(x)$ and also with the analogous quantities $D(x)$ and $S(x)$ defined by the error term $\Delta(x)$.

We will prove the analogues of Theorem B and C as follows.

\begin{theorem}
\label{th:1}
We have for any fixed $\delta > 0$
\beq
\label{eq:2.14}
\omega_W(x) \sim \omega_{D_M} (x) \sim \omega_{S_{M, \delta}}(x) \sim \omega(x) \ \text{ as }\ x \to \infty.
\eeq
\end{theorem}

\begin{theorem}
\label{th:2}
We have for any fixed $\delta > 0$
\beq
\label{eq:2.15}
\log W(x) \sim \log D_M (x) \sim \log S_{M, \delta}(x) \sim \log Z(x).
\eeq
\end{theorem}

\begin{corollary}
\label{cor:1}
We have for any $\delta > 0$
\beq
\label{eq:2.16}
\omega_{D_M}(x) \sim \omega_{S_{M, \delta}}(x) \sim \omega_D(x) \sim \omega_S(x).
\eeq
\end{corollary}

\begin{corollary}
\label{cor:2}
We have for any $\delta > 0$
\beq
\label{eq:2.17}
\log D_M(x) \sim \log S_{M, \delta}(x) \sim \log D(x) \sim \log S(x).
\eeq
\end{corollary}

Corollaries \ref{cor:1}--\ref{cor:2} show that the average and the maximal order of $|M(x)|$
and $|\Delta(x)|$ are both very close to each other and are also equal with high accuracy to the maximal term in the remainder term of the PNT.

\section{The oscillation of $M(x)$. The lower bound for $\vartheta < 1$}
\label{sec:3}

The lower bound \eqref{eq:2.9} is not strong enough for us in general.
However, it establishes for us the lower estimation of $D_M(x)$ in case $\vartheta < 1$.

Let us consider a series of zeros $\varrho_n' = \beta_n' + i\gamma_n' = 1 - \eta_n' + i\gamma_n'$
with
\beq
\label{eq:3.1}
\beta_n' \geq \vartheta - \frac{\varepsilon}{2} \qquad (n = 0, 1, 2, \ldots) \qquad \gamma_n' > 0.
\eeq
For $x > e^{\gamma_n' + 4} \geq (\gamma_n')^{6/\varepsilon}$ we have by \eqref{eq:2.9}
\beq
\label{eq:3.2}
D_M(x) \gg \frac{x^{\beta_n'}}{(\gamma_n')^3} \gg x^{\vartheta - \varepsilon}.
\eeq

Since by $\eta \geq 1 - \vartheta$ we have trivially
\beq
\label{eq:3.3}
\omega(x) = \min_\varrho (\eta \log x + \log \gamma) \geq (1 - \vartheta) \log x,
\eeq
the inequality \eqref{eq:3.2} implies
\begin{align}
\label{eq:3.4}
\omega_{D_M}(x) = \log \frac{x}{D_M(x)} &\leq (1 \! -\! \vartheta + \varepsilon) \log x = (1\! -\! \vartheta) \left(\! 1 + \frac{\varepsilon}{1\! -\! \vartheta}\!\right)\! \log x\\
&\leq \left(1 + \frac{\varepsilon}{1 - \vartheta}\right) \omega(x)
\nonumber
\end{align}
for $x > x_0(\varepsilon, \vartheta)$.
Now \eqref{eq:3.2}--\eqref{eq:3.4} settle the upper bound for $\omega_{D_M}(x)$, i.e., the lower bound for $D_M(x)$ if $\vartheta < 1$.

So, in the rest of this paper we can assume $\vartheta = 1$.
In order to deal with this case, \eqref{eq:2.9}, our theorem from \cite{Pin1982/83} is not sufficient but we can use the method of its proof to reach a much stronger result for zeros with real parts near to $1$.

However, we have to use some strong estimate for $\zeta(s)$ near the boundary line $\sigma = 1$,
\beq
\label{eq:3.5}
\zeta(1 - \eta + it) \ll |t|^{o(\eta)} \log^C(|t| + 2) \ \ \text{ if } \ \eta \to 0.
\eeq

We will use the theorem of Korobov--Vinogradov (for an explicit form see e.g. \cite{For2002})
\beq
\label{eq:3.6}
\zeta(1 - \eta + it) \ll |t|^{C\eta^{3/2}} \log^C(|t| + 2).
\eeq
These types of estimates will enable us to overcome the difficulties of the lower estimate in case of $\vartheta = 1$.

\section{Upper bounds for $M(x)$. Preparation and lemmas}
\label{sec:4}

In this section we will prove two lemmas which give upper bounds for $|1/\zeta(s)|$ in the critical strip on the right from an ``extreme right'' lying zero of $\zeta(s)$.
The first one refers to the case $\vartheta < 1$ and the proof is almost identical to the proof of Theorem 14.2 of Titchmarsh \cite{Tit1951}, valid under the Riemann Hypothesis.
So we completely omit the proof which is based on the Borel--Carath\'eodory theorem and Hadamard's three-circle theorem.

\begin{lemma}
\label{lem:4.1}
\beq
\label{eq:4.1}
\log \zeta(s) = O\left(\log_2 t(\log t)^{(1 - \sigma)/(1 - \vartheta)}\right)
\eeq
\end{lemma}
for
\beq
\label{eq:4.2}
\vartheta + (\log_2 t)^{-1} \leq \sigma \leq 1.
\eeq

This will settle completely the upper estimation of $M(x)$ for $\vartheta < 1$.
We can choose namely the way $J$ of integration until the height $|t| \leq x^2$ on the vertical line
$\text{\rm Re }s = \vartheta + \varepsilon$ and afterwards include the horizontal segments $|t| = x^2$,
$\vartheta + \varepsilon \leq \sigma \leq 2$.

In such a way we obtain from Lemma \ref{lem:4.1} and from the formula on p.\ 316 of \cite{Tit1951, \S 14.7}
\begin{align}
\label{eq:4.3}
M(x) &= \frac{1}{2\pi i} \int\limits_{2 - i x^2}^{2 + ix^2} \frac{x^s}{s\zeta(s)} ds + O(1)\\
&= \frac{1}{2\pi i} \int\limits_J \frac{x^s}{s\zeta (s)} ds + O(1) \nonumber\\
&\ll x^{\vartheta + \varepsilon + \varepsilon} \log x + O(1).\nonumber
\end{align}
This immediately implies
\beq
\label{eq:4.4}
S_{M,\delta}(x) \ll x^{\vartheta + \varepsilon}, \ \ \ D_M(x) \ll x^{\vartheta + \varepsilon},
\eeq
\beq
\label{eq:4.5}
\omega_{S_{M, \delta}}(x) \geq (1 - \vartheta - \varepsilon) \log x + O(1), \ \
\omega_{D_M}(x) \geq (1 - \vartheta - \varepsilon)\log x + O(1).
\eeq

Our next lemma will be used for the case $\vartheta = 1$.

\begin{lemma}
\label{lem:4.2}
Let $t_0 > C$ and suppose that for an $\varepsilon < \varepsilon_0$
\beq
\label{eq:4.6}
\zeta(\sigma + it) \neq 0 \ \ \text{ for } \ \sigma > \sigma_0 = 1 - \eta, \ |t - t_0| \leq \eta.
\eeq
We have then for $\sigma_1 \geq \sigma_0 + \varepsilon \eta$, $|t - t_0| \leq \eta$
\beq
\label{eq:4.7}
\bigl|\log \zeta (\sigma_1 + it_0)\bigr| \leq \eta \log \frac{1}{\varepsilon \eta}\log t_0 + \log \frac1{\varepsilon} \log_2 t_0.
\eeq
\end{lemma}

\begin{proof}
We start with Satz VII.4.1. of \cite{Pra1957}, p. 225 applied in the special case
$\mathcal L(s, \chi) = \zeta(s)$:
\beq
\label{eq:4.8}
\frac{\zeta'}{\zeta}(s) = \sum_{|s - \varrho|\leq 1} \frac1{s - \varrho} - \frac1{s - 1} + O\bigl(\log(|t| + 2)\bigr).
\eeq
We will use it along the segment $s = \sigma + it_0$,
$\sigma_0 + \varepsilon \eta \leq \sigma \leq 1 + \eta $ to estimate
\beq
\label{eq:4.9}
I := \bigl|\log \zeta (\sigma_1 + it_0) - \log \zeta(1 + \eta + it_0)\bigr| =
\biggl| \int\limits_{\sigma_1}^{1 + \eta} \frac{\zeta'}{\zeta} (\sigma + it_0)d\sigma\biggr|.
\eeq

A standard application of the estimate (holding for $t > C$)
\beq
\label{eq:4.10}
\zeta(1 - u + it) \ll t^u \log^C t,
\eeq
together with Jensen's formula yields for $s = 1 + it$ and for any $\dfrac{1}{\log t} \ll r < 2$, say,
\beq
\label{eq:4.11}
\#\{\varrho; |s - \varrho| \leq r\} \ll r \log t + \log_2 t.
\eeq

Let us denote for $\nu$ with $2^\nu \leq 1/\eta$
\beq
\label{eq:4.12}
A_\nu = \bigl\{\varrho; 2^\nu \eta \leq |1 + it_0 - \varrho| < 2^{\nu + 1}\eta\bigr\}, \ \ \nu = 0, 1, 2, \ldots \nu(\eta).
\eeq

Then \eqref{eq:4.11} gives for the sum in \eqref{eq:4.8} for any $s = \sigma + it_0$ with $\sigma \in [\sigma_1, 1 + \eta]$
\beq
\label{eq:4.13}
\sum_{\varrho \in A_\nu} \frac1{s - \varrho} \ll \frac{|A_\nu|}{(2^\nu - 1)\eta} \ll \log t_0 + \frac{\log_2 t_0}{2^\nu \eta} \ \text{ for } \ 1 \leq \nu \leq \nu(\eta),
\eeq
\beq
\label{eq:4.14}
\sum_{\nu = 1}^{\nu(\eta)} \sum_{\varrho \in A_\nu} \frac1{s - \varrho} \ll \log\frac1{\eta} \log t_0 + \frac1{\eta} \log_2 t_0.
\eeq
Hence, the total contribution of these terms to the integral $I$ in \eqref{eq:4.9} is
\beq
\label{eq:4.15}
\ll \eta \log \frac1{\eta} \log t_0 + \log_2 t_0 \ll \eta \log \frac1{\eta} \log t_0 \ \text{ if } \ \eta \gg \frac1{\log t_0}.
\eeq
Finally, the contribution of the remaining $H$ zeros with $|1 + it_0 - \varrho| \leq 2\eta$ is
\beq
\label{eq:4.16}
\ll |H| \int\limits_{\varepsilon\eta}^{2\eta} \frac{du}{u} \ll \log \frac1{\varepsilon} (\eta \log t_0 + \log_2 t_0).
\eeq
This proves Lemma \ref{lem:4.2}, taking into account that (see, e.g., \cite{Tit1951, Theorem 3.5 and Theorem 3.11})
\beq
\label{eq:4.17}
\log \zeta (1 + \eta + it_0) \ll \log_2 t_0.
\eeq
\end{proof}

\section{The Huxley--Hooley--Ramachandra contour}
\label{sec:5}

For obtaining an upper bound of $M(x)$ we shall modify the way of integration to the following contour, which appeared in an unpublished work of Huxley and Hooley and later in a modified form in a work of Ramachandra \cite{Ram1976}.

We take the rectangle $1/2 \leq \sigma \leq 1$, $|t| \leq T :=x^2$ and divide it into equal rectangles of height $40 \log^2 T$ (the smaller rectangles at the ends we ignore) so that the real line cuts into two equal parts one of these rectangles $R_0$.
Let $R_n$ $(n = -n_1, \ldots, n_1)$ be these rectangles.
In a typical rectangle (with $|n| \leq n_1$) we fix a new right side and obtain a new rectangle $R_{n, 0}$ as follows.
Take $R_{n - 1}$, $R_n$, $R_{n + 1}$ whenever all are defined and in the union of these rectangles, pick out a zero $\varrho_n$ with the greatest real part $\beta_n$.
Consider only the right edges of these rectangles with $\sigma_n = \beta_n$ instead of $\sigma = 1$, and join the ends of these edges by horizontal lines.

\smallskip
\noindent
{\bf Case 1.} $\vartheta = 1$. If $I' = \{s'\}$ denotes the contour obtained above, then let us choose our final contour as
\beq
\label{eq:5.1}
I = \bigl\{s;\, \text{\rm Re } s = \text{\rm Re }s' + \varepsilon (1 - \text{\rm Re }s'),
\text{\rm Im }s = \text{\rm Im }s'; s' \in I'\bigr\}.
\eeq
We obtain then by Lemma \ref{lem:4.2}, that is, by \eqref{eq:4.7},
\beq
\label{eq:5.2}
\frac1{\zeta(s)} = O\bigl(|t|^{\varepsilon + \eta \log (1/\varepsilon \eta)}\bigr) \ \text{ if } \
|t| > C_0(\varepsilon), \ s \in I.
\eeq

\smallskip
\noindent
{\bf Case 2.} $\vartheta < 1$. In this case, choosing $\varepsilon < (1 - \vartheta)/2$ and
\beq
\label{eq:5.3}
I = \{s; \, \text{\rm Re }s = \vartheta + \varepsilon\}
\eeq
we obtain by Lemma \ref{lem:4.1}
\beq
\label{eq:5.4}
\frac1{\zeta(s)} = O\bigl(|t|^\varepsilon \bigr) \ \ \text{ for } \ |t| > C_1(\varepsilon), \ \ s \in I.
\eeq

\smallskip
Since we have the crucial estimates \eqref{eq:5.2} and \eqref{eq:5.4} on the integration way, our task will be easy.
Let us denote by $I_v$ the union of the vertical, by $I_h$ the union of the horizontal segments of
$I = I_v \cup I_h$.
To any vertical segment of the contour we can associate a zero (every zero appears at most 3 times) and to any horizontal segment of the contour connecting the points
\beq
\label{eq:5.5}
\varrho^* = \beta^* + \varepsilon (1 - \beta^*) + it \ \ \text{ and } \ \ \widetilde\varrho = \widetilde\beta + \varepsilon(1 - \widetilde \beta) + it
\eeq
we associate the zero with larger real part (so $\varrho^*$ if $\beta^* > \widetilde\beta$, and $\widetilde\varrho$ if $\widetilde \beta > \beta^*$, either one if $\beta^* = \widetilde\beta$),
and denote it simply by $\varrho^+ = \beta + \varepsilon(1 - \beta) + i\gamma = 1 - (1 - \varepsilon)\eta + i\gamma$.

In order to estimate the contribution of $I_\nu$ (and $I_h$) associated with zeros separated from the boundary line by at least a small fixed constant $\delta$ we estimate the following quantities separately.

In such a way we obtain for the integration along with $\eta > \delta$, $J'_{T^*}$ (which will stand for the part of the contour $I$ between $|\text{\rm Im }s| \in [T^*, 2T^*]$) for any $u \leq x$ by Lemmas \ref{lem:4.1}--\ref{lem:4.2} and \eqref{eq:5.2}
\begin{align}
\label{eq:5.6}
M(J'_{T^*},u) &:= \int\limits_{J'_{T^*}} \frac{u^s}{s\zeta(s)} ds + O(x^\varepsilon)\\
&\ll \max_{\varrho^+ \in [T^*/2, 3T^*]} {x^{1 - \eta + \varepsilon \eta}} (T^*)^{\varepsilon + C\eta \log \frac{1}{\varepsilon\eta}}\log T^* \ll x^{1 - \delta/2} \ll xe^{-2\omega(x)}. \nonumber
\end{align}

Namely, in case of $\vartheta = 1$ we must have a sequence of zeros with $\beta_n = 1 - \eta_n \to 1$, i.e. $\eta_n \to 0$ and therefore we must have
\beq
\label{eq:5.7}
\omega(x) = \min_\varrho (\eta \log x + \log \gamma) = o(\log x),
\eeq
since for sufficiently large $x$ values we can find zeros with $\gamma = x^{o(1)}$, $\beta = 1 - o(1)$.
Hence, reversed, the zero which minimizes the expression $\eta \log x + \log \gamma$ must therefore satisfy $\eta \to 0$ as $x \to \infty$.
Hence, for this zero we must have $\eta < \eta_0(\varepsilon)$ for $x > x_0(\varepsilon)$ and consequently, also
\beq
\label{eq:5.8}
C \eta \log(1/\varepsilon \eta) < \varepsilon.
\eeq

Using the density theorem of Carlson (cf.\ \eqref{eq:1.6}) we see that the total length of the contour of $J^{''}_{T^*}$ associated with zeros with $\eta \leq \delta$ is in the range $[T^*/2, 3T^*]$ at most $T^{4\delta + \varepsilon}$, hence
\beq
M(J^{''}_{T^*}, u) \ll x^{1 - \eta + \varepsilon \eta}(T^*)^{4\delta + 2\varepsilon - 1}.
\label{eq:5.9}
\eeq
The contribution of integrals on $I_h$ can be incorporated into those of $I_\nu$.

Choosing $\delta \leq \varepsilon/4$ and
repeating this for $T^* = T_\nu = x^2/2^\nu$ we obtain by $\eta \gg \frac1{\log \gamma}$, consequently by $\eta \log x + \log \gamma \gg \sqrt{\log x}$,
\begin{align}
\label{eq:5.10}
\max_{u \leq x} |M(u)| &\leq \max_{\varrho; ~0 < \gamma \leq x^2} \frac{x}{x^{\eta(1 - \varepsilon)}
\gamma^{1 - 5\varepsilon/2}} \cdot \log^2 x
\\ & =\frac{x \log x}{\min_{\varrho; ~0 < \gamma \leq x^2} \exp\left( \eta(1 - \varepsilon)\log x + (1 - 5\varepsilon/2)\log \gamma\right)}
\nonumber
\\ &\leq \frac{x \log x}{e^{\omega(x)(1 - 5\varepsilon/2)}} \leq  \frac{x }{e^{\omega(x)(1 - 3\varepsilon)}}. \nonumber
\end{align}
Hence, by the definition of $\omega_{S_{M,\delta}}(x)$ and $\omega_{D_M}(x)$ (cf.\ \eqref{eq:2.12}--\eqref{eq:2.13})
\beq
\label{eq:5.11}
\omega_{D_M}(x) \geq (1 - 3\varepsilon) \omega(x), \ \ \ \omega_{S_{M,\delta}} (x) \geq (1 - 3\varepsilon) \omega(x).
\eeq

\section{Proofs of the theorems, I. Upper and lower bounds for $|M(x)|$}
\label{sec:6}

Although the proofs of
\beq
\label{eq:6.1}
\log W(x) \sim \log Z(x) = \log \frac{x}{e^{\omega(x)}} \sim \vartheta \log x
\eeq
and (see \eqref{eq:2.10}--\eqref{eq:2.11}) the finer relation
\beq
\label{eq:6.2}
\omega_W(x) \sim \omega_Z(x) = \omega(x)
\eeq
are contained in Theorems B and C, we will include their proofs here, for the sake of completeness.

Firstly, for $1/2 \leq \vartheta < 1$ by $N(T + 1) - N(T) \ll \log T$ we have for any $\varepsilon > 0$ (cf.\ \eqref{eq:1.4})
\beq
\label{eq:6.3}
Z(x) \leq W(x) \ll \sum_{\beta \geq 1/2, \gamma \leq x} \frac{x^\beta}{\gamma} \ll x^\vartheta \log^2 x \ll_\varepsilon x^{\vartheta + \varepsilon}.
\eeq

Secondly, if we take a sequence of zeros with $\beta_n \to \vartheta$, then we can suppose $\beta_n > \vartheta - \varepsilon/2$ for $n \geq n_0$.
Thus, for
\beq
\label{eq:6.4}
x > \gamma_{n_0}^{2/\varepsilon} \ \Longleftrightarrow\ \gamma_{n_0} < x^{\varepsilon/2}
\eeq
we have, together with \eqref{eq:6.3},
\beq
\label{eq:6.5}
x^{\vartheta + \varepsilon} \gg W(x) \geq Z(x) \geq \frac{x^{\beta_{n_0}}}{\gamma_{n_0}} \geq x^{\vartheta - \varepsilon/2 - \varepsilon/2} = x^{\vartheta - \varepsilon},
\eeq
which proves \eqref{eq:6.1}.

The relation $\omega_W(x) \leq \omega_Z(x) = \omega(x)$ being trivial for any $\vartheta$
we have to show only $\omega_W(x) \geq (1 - \varepsilon) \omega(x)$ for any $\varepsilon > 0$, that is, a good upper bound for $W(x)$ as a function of zeta-zeros.
We can divide the zeros with $\beta \geq 1 - \varepsilon$, $\gamma \leq x$ into $\mathcal L^2 = \lceil C \log x / \log 2\rceil^2$ classes $A_{\kappa, \nu}$ according to
\beq
\label{eq:6.6}
\beta = 1 - \eta, \ \ \eta \in \left\lceil \frac{\kappa - 1}{2\mathcal L}, \frac{\kappa}{2\mathcal L}\right\rceil, \ \  \lceil 2\varepsilon \mathcal L\rceil > \kappa \geq 2,
\eeq
\beq
\label{eq:6.7}
\gamma \in (T_\nu, 2T_\nu], \ \ T_\nu = \frac{x}{2^\nu}, \ \ \nu = 1,\ldots, \mathcal L.
\eeq

Using Carlson's density theorem \eqref{eq:1.6} we have for the contribution of the class $A_{\kappa, \nu}$ to $W(x) = \sum \frac{x^\beta}{\gamma}$
\beq
\label{eq:6.8}
W(\kappa, \nu) \ll \frac{x^{1 - (\kappa - 1)/2\mathcal L}}{T_\nu} T_\nu^{4\kappa/2\mathcal L} \log^C T_\nu.
\eeq

Hence, cutting $W(x)$ into two parts, $W_0(x)$ and $W_1(x)$:
\beq
\label{eq:6.9}
W_0(x) := \sum_{\substack{\beta \geq 1 - \varepsilon\\
\gamma \leq x}} \frac{x^\beta}{\gamma} \ll (\log x)^2 \max W(\kappa, \nu) \ll (\log x)^{C + 2} \max_{\substack{\varrho\\ \beta \geq 1 - \varepsilon}} \frac{x^\beta}{\gamma^{1 - 4\varepsilon}}.
\eeq

We have also, by $N(T + 1) - N(T) = O(\log T)$,
\beq
\label{eq:6.10}
W_1(x) := \sum_{\substack{\beta \leq 1 - \varepsilon\\
\gamma \leq x}} \frac{x^\beta}{\gamma} \ll x^{1 - \varepsilon} \log^2 x.
\eeq

Since in case $\vartheta = 1$, the first upper bound \eqref{eq:6.9} is dominant as $x \to \infty$ we obtain from \eqref{eq:6.9}--\eqref{eq:6.10} for the case $\vartheta = 1$
\beq
\label{eq:6.11}
\log \frac{x}{W(x)} \geq \min_\varrho(\eta \log x + (1 - 4\varepsilon)\log \gamma) - (C + 2)\log_2 x \geq (1 - 5\varepsilon)\omega(x),
\eeq
since
\beq
\label{eq:6.12}
\eta \log x + \log \gamma \gg \frac{\log x}{\log \gamma} + \log \gamma \geq 2\sqrt{\log x}.
\eeq

In case of $\vartheta < 1$, by $\omega(x) \leq (1 - \vartheta + \varepsilon) \log x$ we get even easier from \eqref{eq:6.3}
\begin{align}
\label{eq:6.13}
\omega_W(x) = \log \frac{x}{W(x)} &\geq (1 - \vartheta - \varepsilon) \log x \geq (1 - \vartheta + \varepsilon) \left(1 - \frac{3\varepsilon}{1 - \vartheta}\right) \log x\\
&\geq \left(1 - \frac{3\varepsilon}{1 - \vartheta}\right) \omega(x) \nonumber
\end{align}
if $\varepsilon < \varepsilon_0(\vartheta)$.

To continue with $\omega_D(x)$ we distinguish again two cases.

\smallskip
\noindent
{\bf Case 1.} $\vartheta = 1$. Let us consider now the zero $\beta_0 = 1 - \eta_0 + i\gamma_0$ for which the minimum is reached in $\omega(x)$.
Since
\beq
\label{eq:6.14}
\omega(x) \leq (1 - \vartheta + \varepsilon^2) \log x = \varepsilon^2 \log x \ \ \text{ for } \ x > x_0(\varepsilon)
\eeq
we must have then $\log \gamma_0 \leq \varepsilon^2 \log x$, hence
\beq
\label{eq:6.15}
\gamma_0 \leq x^{\varepsilon^2} \ \ \text{ and } \ \ \eta_0 \leq \varepsilon^2.
\eeq

\smallskip
We continue with a result of independent interest which substitutes \eqref{eq:2.9} with a stronger result for $\text{\rm Re }\varrho$ near to $1$, needed for the case $\vartheta = 1$.

\begin{theorem}
\label{th:3}
Let $\varrho_0 = \beta_0 + i\gamma_0 = 1 - \eta_0 + i\gamma_0$ a zero of $\zeta(s)$ with $\gamma_0 > 0$.
Let $Y$ be sopme large parameter, $Y_0 = Ye^3$, $0<\eta_0 < 1/10$, $2\eta_0 < \kappa < \frac12 - 2\eta_0$.
Then
\beq
\label{eq:6.15}
\int\limits_1^{Y_0} \frac{|M(x)|dx}{x^{1 - \kappa + \beta_0}/\gamma_0} \gg \frac{\kappa Y^\kappa}{\gamma_0^{C(\eta_0 + \kappa)^{3/2}}(\log \gamma_0 Y)^C}
\eeq
\end{theorem}

\begin{proof}
Let
\beq
\label{eq:6.16}
\lambda = \log Y, \ \ \ g(s) := \frac{(s + \varrho_0 - \kappa - 1)\zeta(s + \varrho_0 - \kappa)}{(s - \kappa)(s + 1)^4},
\eeq
\beq
\label{eq:6.17}
r_\lambda(H) := \frac1{2\pi i} \int\limits_{(3)} e^{s^2/\lambda + Hs} g(s) ds.
\eeq

We note that $g(s)$ is analytic for $\sigma > -1$, since it is regular at $s = 1 - \varrho_0 + \kappa$ and $s = \kappa$.
We start with the formula (valid for $\sigma > 1$)
\beq
\label{eq:6.18}
\int\limits_1^\infty \frac{M(x)}{x^{s + 1}} dx = \frac1{s\zeta(s)}.
\eeq
Since we will prove (cf.\ \eqref{eq:6.20}) that $r_\lambda(\lambda - \log x)$ decreases exponentially for $x > Y_0$ we may define the integral and later interchange the order of integrations below.
Let
\begin{align}
\label{eq:6.19}
V :&= \int\limits_1^\infty \frac{M(x)}{x^{1 + \varrho_0 - \kappa}} r_\lambda (\lambda - \log x)dx\\
&= \frac1{2\pi i} \int\limits_{(3)} e^{s^2/\lambda + \lambda s} g(s)\int\limits_1^\infty \frac{M(x)}{x^{1 + \varrho_0 - \kappa + s}}dx ds \nonumber\\
&= \frac1{2\pi i} \int\limits_{(3)} e^{s^2/\lambda + \lambda s} \frac{\zeta(s + \varrho_0 - \kappa)(s + \varrho_0 - \kappa - 1) ds}{\zeta(s + \varrho_0 - \kappa)(s + \varrho_0 - \kappa)(s + 1)^4(s - \kappa)}\nonumber\\
&= \left(1 - \frac1{\varrho_0}\right)(1 + \kappa)^{-4} e^{\kappa^2/\lambda}Y^\kappa + O(Y^{-1/3})\nonumber
\end{align}
if we translate the last integration to the vertical line $\sigma = -1/3$.

In order to see the decay of $r_\lambda(H)$ for $H \to -\infty$, that is, for $x = e^{\lambda - H} \to \infty$ we can translate the way of integration in \eqref{eq:6.17} to the vertical line $\sigma = \lambda$:
\begin{align}
\label{eq:6.20}
r_\lambda(H) &= \frac1{2\pi i} \int\limits_{(\lambda)} e^{s^2/\lambda + Hs} g(s)ds \ll \int\limits_{-\infty}^\infty e^{\lambda - t^2/\lambda + H \lambda}dt\\
&\ll \sqrt{\lambda} e^{(H + 1)\lambda} = \sqrt{\lambda} e^{\lambda^2 + \lambda} x^{-\lambda}.\nonumber
\end{align}
Consequently, the part of the integral $V$ with $x \geq Y_0 \Leftrightarrow H \leq - 3$ is
\begin{align}
\label{eq:6.21}
&\int\limits_{Y_0}^\infty \frac{M(x)}{x^{1 + \varrho_0 - \kappa}} r_\lambda (\lambda - \log x)dx \ll \sqrt{\lambda}e^{\lambda^2 + \lambda} \int\limits_{e^{\lambda + 3}} \frac{dx}{x^\lambda}\\
&= \sqrt{\lambda} e^{\lambda^2 + \lambda} \cdot \frac1{\lambda - 1} e^{-(\lambda + 3)(\lambda - 1)} \ll e^{-\lambda} = Y^{-1}.\nonumber
\end{align}

Further, integrating on $\sigma = 0$ we obtain for $H = \lambda - \log x \in [-3, \lambda]$, that is, for $x \leq Y_0$
\begin{align}
\label{eq:6.22}
r_\lambda(H) &\ll \int\limits_{-2\lambda}^{2\lambda} e^{-t^2/\lambda} \frac{\max\limits_{|t| \leq 2\lambda} \bigl|\zeta(\beta_0 - \kappa + i(\gamma_0 + t))\bigr|(\gamma_0 + 2\lambda)}{\kappa(t^2 + 1)^2}dt\\
&\ll \frac1{\kappa} (\gamma_0 + \lambda)^{C(\eta_0 + \kappa)^{3/2}} \log^C(\gamma_0 + \lambda),
\nonumber
\end{align}
using the Korobov-Vinogradov type estimate \eqref{eq:3.6}.
Now, $|M(x)| \leq x$, \eqref{eq:6.19}, \eqref{eq:6.21} and \eqref{eq:6.22} prove Theorem \ref{th:3}.

\bigskip
We can estimate trivially the integration on $[1, Y^{1 - \varepsilon'}]$ as
\beq
\label{eq:6.23}
\int\limits_1^{Y^{1 - \varepsilon'}} \frac{|M(x)|}{x^{1 - \kappa + \beta_0}/\gamma_0} dx
\leq \int\limits_1^{Y^{1 - \varepsilon'}} \frac{dx}{x^{1 - \eta_0 - \kappa}/\gamma_0}
\leq \frac{\gamma_0 Y^{(1 - \varepsilon')(\eta_0 + \kappa)}}{\eta_0 + \kappa}.
\eeq
This implies by Theorem \ref{th:3} the following

\setcounter{corollary}{2}
\begin{corollary}
\label{cor:3}
Suppose beyond the conditions of Theorem \ref{th:3}
\beq
\label{eq:6.24}
\varepsilon' = \varepsilon/8, \ \ \sqrt{\varepsilon} \geq \kappa \geq \varepsilon, \ \ \eta_0 \leq \varepsilon^2/100, \ \ \gamma_0 \leq Y^{\varepsilon^2/100}, \ \ Y > Y_0(\varepsilon).
\eeq
Then we have
\beq
\label{eq:6.25}
\int\limits_{Y^{1 - \varepsilon/8}}^Y \frac{|M(x)|dx}{x^{1 - \kappa + \beta_0}/\gamma_0}
\gg \frac{\kappa Y^\kappa}{\gamma_0^{C(\eta_0 + \kappa)^{3/2}}(\log \gamma_0 Y)^C}.
\eeq
\end{corollary}

\begin{proof}
Follows from \eqref{eq:6.15} and \eqref{eq:6.23} by \eqref{eq:6.24}, in view of
\beq
\label{eq:6.26}
\kappa - (1 - \varepsilon')(\eta_0 + \kappa) \geq \varepsilon'\kappa - \eta_0 \geq \frac{\varepsilon^2}{10}.
\eeq
\end{proof}

In order to obtain a lower estimation for the weighted average of $|M(x)|$ in some interval $[Y_1, Y_0]$ we have to use \eqref{eq:6.15} and prove an upper bound for the weighted average of $|M(x)|$ in $[1, Y_1]$.
To obtain the upper bound we can use the results of Section \ref{sec:5}, namely, \eqref{eq:5.10}.

\section{Proofs of the theorems, II. Lower bound for the average of $|M(x)|$}
\label{sec:7}

We settled already the easier case $\vartheta < 1$ and the upper bound for $|M(x)|$ in the case $\vartheta = 1$ (cf.\ \eqref{eq:5.10}--\eqref{eq:5.11}) for all values of $x$.
So our task is now to prove a lower bound for the average value of $|M(x)|$.
This will automatically yield a lower bound for the maximum of $|M(x)|$ in an interval of type $[Y^{1 - \delta}, Y]$ for any $\delta > 0$.
Theorem \ref{th:3} points in this direction already.
Interestingly, in order to obtain the desired lower estimate for the average of $|M(x)|$ we have to use our above mentioned upper estimate of $|M(x)|$, too.

Let us investigate first for any $x \in [Y^{1 - \varepsilon'}, Y]$
\beq
\label{eq:7.1}
A(x) := \frac{Z(x)}{Z_0(x)} := \exp \left(-\eta_x \log x - \log \gamma_x + \eta_0 \log x + \log \gamma_0\right),
\eeq
where $\varrho_0 = 1 - \eta_0 + i\gamma_0$ is a zero for which the maximum $Z(Y)$ is attained and $\varrho_x = 1 - \eta_x + i\gamma_x$ is a zero for which $Z(x)$ is attained, whereas $Z_0(x):=x^{1-\eta_0}/\gamma_0$ is a new quantity, where the $\zeta$-zero $\varrho_0$ is fixed and only $x$ changes. We may suppose $\eta_0, \eta_x < \varepsilon/100$.

By the definition we have
\beq
\label{eq:7.2}
\eta_0 \log Y + \log \gamma_0 \leq \eta_x \log Y + \log \gamma_x,
\eeq
so
\beq
\label{eq:7.3}
\log \gamma_0 - \log \gamma_x \leq \log Y(\eta_x - \eta_0).
\eeq
Hence, from \eqref{eq:7.1}
\beq
\label{eq:7.4}
A(x) \leq \exp \bigl((\eta_x - \eta_0)(\log Y - \log x)\bigr).
\eeq
If $\eta_0 \geq \eta_x$, then $A(x) \leq 1$. If $\eta_0 \leq \eta_x$, then
\beq
\label{eq:7.5}
A(x) \leq \exp\left(\eta_x \cdot \log x \frac{\varepsilon'}{1 - \varepsilon'}\right) \leq \exp(2 \varepsilon'\eta_x \log x) \leq e^{2\varepsilon' \omega(x)}.
\eeq
Consequently,
\beq
\label{eq:7.6}
\frac{|M(x)|}{Z_0(x)} = \frac{|M(x)|}{Z(x)} \cdot \frac{Z(x)}{Z_0(x)} = \frac{|M(x)|}{Z(x)} \cdot A(x) \leq \frac{|M(x)|}{Z(x)} e^{2\varepsilon'\omega(x)}.
\eeq
Since $\omega(x)$ is increasing as a function of $x$, applying \eqref{eq:5.10} with $\varepsilon'$ in place of $\varepsilon$ we obtain by $Z(x) = xe^{-\omega(x)}$
\beq
\label{eq:7.7}
\frac{|M(x)|}{x^{\beta_0}/\gamma_0} =
\frac{|M(x)|}{Z_0(x)} \leq \frac{|M(x)|}{Z(x)} e^{2\varepsilon'\omega(x)} \ll e^{7\varepsilon'\omega(x) + 2\varepsilon'\omega x} = e^{9\varepsilon'\omega(x)}\leq e^{\varepsilon \omega(Y)}
\eeq
if we choose $\varepsilon' = \varepsilon/9$.
Let us choose further $\kappa = \varepsilon^{1/2}$.
Therefore we have by \eqref{eq:7.7} with $Y^* \in [Y^{1 - \varepsilon/8}, Y]$, $\omega_0 = \omega(Y)$
\beq
\label{eq:7.8}
\int\limits_{Y^{1 - \varepsilon/8}}^{Y^*} \frac{|M(x)|dx}{x^{1 - \kappa + \beta_0}/\gamma_0}
\ll e^{\varepsilon\omega_0} \int\limits_{Y^{1 - \varepsilon/8}}^{Y^*} \frac{dx}{x^{1 - \kappa}}
\leq \frac{e^{\varepsilon \omega_0}}{\kappa} (Y^*)^\kappa.
\eeq
On the other hand by Corollary \ref{cor:3} we have (as its conditions in \eqref{eq:6.24} are satisfied)
\beq
\label{eq:7.9}
\int\limits_{Y^{1 - \varepsilon/8}}^Y \frac{|M(x)|dx}{x^{1 - \kappa + \beta_0}/\gamma_0} \geq
c^* \frac{\kappa Y^\kappa}{\gamma_0^{C\kappa^{3/2}}(\log Y)^C}=:J_0.
\eeq

Now, we will use the following easy (in fact, trivial) auxiliary consideration. Let $g(x)>0$ be a positive decreasing function (weight) on an interval $[U,V]$, and denote $J :=J(F):= \int\limits_U^V F(x)g(x)dx$ for any integrable function $F(x)$.  Further, let $f(x)$ be a positive and continuous function on $[U,V]$, and let $J_0$ with $0\le J_0 \le J(f)=\int\limits_U^V f(x)g(x)dx$ be given.
Consider the extremal problem of minimizing $H :=H(F):= \int\limits_U^V F(x)dx$ under the condition that the unknown non-negative integrable function $0 \le F(x) \leq f(x)$ satisfies $J(F)\ge J_0$. Then $H(F)$ is minimal if and only if there exists a point $x^* \in [U,V]$ such that $F(x) = f(x)$ for $x < x^*$ and $F(x) = 0$ for $x > x^*$.

Let us apply this argument (cf.\ \eqref{eq:7.7}) with $(\omega_0 = \omega(Y))$
\beq
\label{eq:7.10}
F(x) = |M(x)| \leq Z_0(x)e^{\varepsilon \omega_0} = f(x), \ \  g(x) = \frac1{x^{1 - \kappa} Z_0(x)}.
\eeq
Let us calculate (estimate) the value $Y^*$ for which
\beq
\label{eq:7.11}
\int\limits_{Y^{1 - \varepsilon/8}}^{Y^*} f(x)g(x)dx = \int\limits_{Y^{1 - \varepsilon/8}}^{Y^*} \frac{e^{\varepsilon \omega_0}}{x^{1 - \kappa}} dx \sim \frac{e^{\varepsilon\omega_0}(Y^*)^\kappa}{\kappa} = c^*\frac{\kappa Y^\kappa}{\gamma_0^{C\kappa^{3/2}}(\log Y)^C}.
\eeq
Hence,
\beq
\label{eq:7.12}
\left(\frac{Y}{Y^*}\right)^\kappa \sim \frac{1}{c^*\kappa^2} e^{\varepsilon\omega_0} \gamma_0^{C\kappa^{3/2}}(\log Y)^C.
\eeq

We have
\beq
\label{eq:7.13}
\omega_0 = \eta_0 \log Y + \log \gamma_0 \gg \frac{\log Y}{\log \gamma_0} + \log \gamma_0 \gg \sqrt{\log Y},
\eeq
so by $\gamma_0 \leq e^{\omega_0}$ we obtain from \eqref{eq:7.12} and $\kappa = \sqrt{\varepsilon}$
\beq
\label{eq:7.14}
\frac{Y}{Y^*} \ll_\varepsilon e^{c \varepsilon^{1/4}\omega_0}(\log Y)^C \ll_\varepsilon e^{\varepsilon^{1/5}\omega_0}.
\eeq
This implies
\begin{align}
\label{eq:7.15}
\int\limits_{Y^{1 - \varepsilon/8}}^Y |M(x)|dx &\geq \int\limits_{Y^{1 - \varepsilon/8}}^{Y^*} f(x) dx = \int\limits_{Y^{1 - \varepsilon/8}}^{Y^*} e^{\varepsilon \omega_0}Z_0(x)dx\\
&\sim \frac{e^{\varepsilon\omega_0}(Y^*)^{2 - \eta_0}}{(2 - \eta_0)\gamma_0} \gg_\varepsilon \frac{Y^{2 - \eta_0}}{\gamma_0} e^{-2\varepsilon^{1/5}\omega_0}.
\nonumber
\end{align}
Consequently
\beq
\label{eq:7.17}
D_M(Y) = \frac1{Y} \int\limits_{1}^Y |M(x)|dx \gg_\varepsilon Z(Y)e^{-2\varepsilon^{1/5} \omega(Y)},
\eeq
which furnishes the assertion since $2\varepsilon^{1/5}$ is arbitrarily small.
\end{proof}

\bigskip
{\small
\noindent
J\'anos Pintz\\
HUN-REN Alfr\'ed R\'enyi Institute of Mathematics \\
Budapest, Re\'altanoda u. 13--15\\
H-1053 Hungary\\
e-mail: pintz@renyi.hu}

\end{document}